\documentclass[10pt]{article}
\font\smallit=cmti10
\font\smalltt=cmtt10

\usepackage{amssymb,latexsym,amsmath,epsfig,amsthm}
\makeatletter

\renewcommand\section{\@startsection {section}{1}{\z@}
{-30pt \@plus -1ex \@minus -.2ex}
{2.3ex \@plus.2ex}
{\normalfont\normalsize\bfseries\boldmath}}

\renewcommand\subsection{\@startsection{subsection}{2}{\z@}
{-3.25ex\@plus -1ex \@minus -.2ex}
{1.5ex \@plus .2ex}
{\normalfont\normalsize\bfseries\boldmath}}

\renewcommand{\@seccntformat}[1]{\csname the#1\endcsname. }

\makeatother

\newtheorem{theorem}{Theorem}

\theoremstyle{definition}

\newtheorem{conjecture}{Conjecture}

\begin{document}

\begin{center}
\uppercase{\bf \boldmath Sparse Admissible Sets and a Problem of Erd\H{o}s and Graham}
\vskip 20pt
{\bf Desmond Weisenberg}\\
{\smallit Mathematical Institute, University of Oxford, United Kingdom}\\
{\tt desmondweisenberg@gmail.com} \\
\end{center}
\vskip 20pt

\centerline{\bf Abstract}

\noindent
Erd\H{o}s and Graham asked whether any sparse enough admissible set of natural numbers can be translated into a subset of the primes. By using a greedy construction involving powers of primitive roots, we prove that there exist arbitrarily sparse infinite admissible sets that cannot be translated into a subset of the primes, thus answering this question in the negative. We present three additional constructions as well.

\pagestyle{myheadings}
\markright{\smalltt INTEGERS: 24 (2024)\hfill}
\thispagestyle{empty}
\baselineskip=12.875pt
\vskip 30pt

\section{A Problem of Erd\H{o}s and Graham}

Define a set $A \subseteq \mathbb{N}$ as \emph{admissible} if there does not exist a prime $p$ such that $A$ contains at least one element in each residue class mod $p$. For finite admissible $A$, the Hardy-Littlewood $k$-tuple conjecture states that there are infinitely many $n \in \mathbb{N}$ such that $A + n$ is contained in the primes. (More precisely, the conjecture predicts a certain asymptotic distribution of such $n$, but even the infinitude thereof is still an open question.) This is a major unsolved problem in number theory.

Erd\H{o}s and Graham have asked a similar question for infinite sets: if a sequence of numbers that forms an admissible set grows rapidly enough (or, equivalently, if an admissible set is sparse enough), must it have a translation that is contained in the primes? (Here, the primes are only taken to be positive.) They asked this in their book \cite[p. 85]{ErdosBook}, and it is also listed as problem 429 on Thomas Bloom's website erdosproblems.com \cite{ErdosProblems}. This question can be stated more formally:

\begin{conjecture}[\bf{Erd\H{o}s and Graham}]
\label{ErdosConjecture}
There is a non-decreasing, unbounded function $f: \mathbb{N} \to \mathbb{Z}_{\geq 0}$ such that if $A \subseteq \mathbb{N}$ is admissible and $|A \cap \{1, \dots, N\}| \leq f(N)$ for all $N$, then there exists $n \in \mathbb{Z}$ such that $A + n$ is contained in the primes.
\end{conjecture}

It turns out that this conjecture is false; that is, there exist arbitrarily sparse infinite admissible sets that cannot be translated into a subset of the primes, and there is actually a very straightforward construction for these arbitrarily sparse sets. First, we turn to a brief discussion of primitive roots. Artin's primitive root conjecture states that every integer $a \in \mathbb{Z}$ that is not -1 or a square number is a primitive root mod $p$ for infinitely many primes $p$. Though Artin's conjecture is still unsolved, many partial results are known; in particular, while we do not have an explicit example, we know that there exists at least one positive integer $a$ where this statement is true. (See \cite{Artin1} and \cite{Artin2} for much stronger results in this vein.) Fix such a positive integer $a$, and let $S = \{a^k : k \in \mathbb{N}\}$. Clearly, $S$ is admissible, and therefore all of its subsets are admissible as well. Also, let $p_1, p_2, p_3, \dots$ be the infinite sequence of primes that have $a$ as a primitive root.

We answer the question of Erd\H{o}s and Graham in the negative by proving that $S$ has arbitrarily sparse subsets that cannot be translated into a subset of the primes.

\begin{theorem}
\label{ErdosResult}
Let $f: \mathbb{N} \to \mathbb{Z}_{\geq 0}$ be non-decreasing and unbounded. Then there exists a subset $A \subseteq S$ such that $|A \cap \{1, \dots, N\}| \leq f(N)$ for all $N$ and such that there does not exist an $n \in \mathbb{Z}$ such that $A + n$ is contained in the primes. In particular, Conjecture \ref{ErdosConjecture} is false.
\end{theorem}

\begin{proof}
Observe that for each prime $p_i$ that has $a$ as a primitive root, $S$ has infinitely many elements in each nonzero residue class mod $p_i$. Using this, we can construct $A$ in a greedy manner as follows: first, add two members of each nonzero residue class of $p_1$, then add two members of each nonzero residue class of $p_2$, and so on. Only add elements to $A$ that are greater than all previous elements, and only add elements that are large enough to respect the sparsity condition; more formally, the $m$th element $e_m$ to be added to $A$ must always be large enough so that $f(e_m) \geq m$.

Clearly, $A$ satisfies the sparsity requirement, and for each $p_i$, the set $A$ has at least two members of each nonzero residue class mod $p_i$. We now prove that there is no $n \in \mathbb{Z}$ such that $A + n$ is contained in the primes. For the sake of contradiction, suppose such an $n$ exists. We claim that $n$ is a multiple of every $p_i$; this is true because if $n \not\equiv 0 \pmod{p_i}$, then $A$ has at least two elements congruent to $-n \pmod{p_i}$, so $A + n$ will therefore have at least two elements congruent to $0 \pmod{p_i}$ and thus cannot be contained in the primes. As such, $n$ is a multiple of each $p_i$; that is, $n$ has infinitely many prime factors. The only way this is possible is if $n = 0$, which would mean that $A$ itself is contained in the primes. However, $A$ has infinitely many powers of $a$, so this is clearly not the case.
\end{proof}

\section{Further Constructions}

The above construction was the first one discovered by the author, and it was on this basis that Bloom first marked the problem as ``solved'' on erdosproblems.com. However, since the discovery and uploading of the preprint of that solution, the author has realized that it is in fact possible to solve this problem without using the previously cited result on primitive roots. This makes the solution even more elementary. We demonstrate this by presenting three more constructions of arbitrarily sparse infinite subsets of $\mathbb{N}$ that are admissible yet cannot be translated into a subset of the primes.

Before we present the second construction, observe that the sets $A$ from the first construction were unable to be translated into the primes because they satisfied two properties. First, $A$ had at least one non-prime. Second, there were infinitely many primes $p$ such that $A$ had at least two elements in every nonzero residue class mod $p$. By the logic of the proof of Theorem \ref{ErdosResult}, any set with these two properties cannot be translated into a subset of the primes.

We now present our second construction: for any non-decreasing, unbounded function $f: \mathbb{N} \to \mathbb{Z}_{\geq 0}$, there is an even more ``direct'' greedy way to build an infinite admissible set $B \subseteq \mathbb{N}$ such that $|B \cap \{1, \dots, N\}| \leq f(N)$ for all $N$ and such that there does not exist an $n \in \mathbb{Z}$ such that $B + n$ is contained in the primes. First, add a composite number (or 1, if possible) to $B$ that is large enough to respect the sparsity condition. Then choose some prime $p_1$ larger than the initial number, and add two elements to each nonzero residue class mod $p_1$. Each element added must be large enough to respect the sparsity condition. Furthermore, each element must preserve admissibility --- the Chinese Remainder Theorem guarantees that such an element can always be chosen. Then choose a prime $p_2$ larger than all existing elements of $B$, add two elements from each nonzero residue class mod $p_2$ as described, and keep repeating this process. This completes the second construction.

The first two constructions utilize a set having the property of having at least one non-prime and having at least two elements in each nonzero residue class mod $p$ for infinitely many primes $p$. However, it turns out that a slightly different condition also suffices to show that a set cannot be translated into a subset of the primes. Namely, suppose $C \subseteq{N}$ is nonempty and has the property that for all primes $p$, no residue class mod $p$ has exactly one element in $C$. This also guarantees that $C$ cannot be translated into a subset of the primes. For the sake of contradiction, suppose $C$ has this property and $C + n$ is contained in the primes. Then if we fix $c_1 \in C$, we have that $c_1 + n = p$ for some prime $p$, so $c_1 \equiv -n \pmod{p}$. By our hypothesis, there exists $c_2 \in C$ such that $c_2 \not= c_1$ and $c_2 \equiv -n \pmod{p}$, so $c_2 + n$ is a multiple of $p$ that is not $p$ itself and is therefore not prime. This contradicts that $C + n$ is contained in the primes.

Using this, we demonstrate a third construction that is very similar to the first construction. In the first construction, we proved that if $a$ is a positive integer that is a primitive root for infinitely many primes, then the set of powers of $a$ has arbitrarily sparse infinite subsets that cannot be translated into a subset of the primes. Now, we prove that this is actually true for any integer $c \geq 2$, even if $c$ is a counterexample to Artin's primitive root conjecture (if one exists) or a square number. (Clearly, any subset of powers of $c$ is also admissible, making it a solution to the problem.) Construct a set $C \subseteq \mathbb{N}$ as follows: for all primes $p$ that do not divide $c$, add to $C$ at least two powers of $c$ in each residue class mod $p$ where powers of $c$ exist. This works even if there are not powers of $c$ in every nonzero residue class mod $p$. It is clear that these sets $C$ can be made arbitrarily sparse, that they are admissible, and  for all primes $p$, no residue class mod $p$ has exactly one element in $C$. As such, this is also a construction of arbitrarily sparse infinite admissible sets that cannot be translated into a subset of the primes.

For our last construction, suppose $D \subseteq \mathbb{N}$ has the property that for all integers $n \geq 2 - \text{min}(D)$, there exists $d_n \in D$ such that $d_n + n$ is not prime. It clearly follows that there does not exist $n \in \mathbb{Z}$ such that $D + n$ is contained in the primes. We show that it is possible to construct arbitrarily sparse infinite admissible sets with this property.

For any non-decreasing, unbounded function $f: \mathbb{N} \to \mathbb{Z}_{\geq 0}$, we construct $D$ as follows: first, add a composite number (or 1, if possible) to $D$ that is large enough to respect the sparsity condition. Denote this number $d_0$. For every positive integer $n$ (as well as every integer $n \in [2 - d_0, 0)$ if $d_0 > 1$), add an integer $d_n$ to $D$ such that $d_n$ is larger than all previous elements that $D$ has, $d_n$ is large enough to respect the sparsity condition, $d_n$ preserves the admissibility of $D$, and $d_n + n$ is not prime. The Chinese Remainder Theorem and the non-existence of an infinite arithmetic progression of primes guarantees that this is always possible. Once again, the sets $D$ are arbitrarily sparse infinite admissible sets that cannot be translated into a subset of the primes.

\vskip 20pt\noindent {\bf Acknowledgement.} The author thanks Dr.~Thomas Bloom for reviewing an earlier draft of this paper as well as for being an attentive advisor in the Oxford mathematics master's program.

\end{document}